\documentclass[12pt]{amsart}
\usepackage{dsfont}
\usepackage{amsmath}
\usepackage{amsthm}
\usepackage{amssymb}
\usepackage{multirow}
\usepackage{mathtools}
\usepackage{indentfirst}  
\usepackage{graphicx}
\usepackage{color}
\usepackage{hyperref}

\usepackage{comment}

\newtheorem{theorem}{Theorem}
\newtheorem{lemma}{Lemma}

\newtheorem{conjecture}{Conjecture}
\theoremstyle{remark}
\newtheorem{remark}{Remark}

\newcommand{\Z}{\mathbb{Z}}
\newcommand{\R}{\mathbb{R}}
\newcommand{\Hh}{\mathbb{H}}
\newcommand{\SL}{\mathrm{SL}}

\renewcommand{\pmod}[1]{\mkern3mu(\mathrm{mod}\ #1)}

\title{On the hyperbolic prime number theorem}
\author{Alisa Sedunova}
\email{alisa.sedunova@gmail.com}

\begin{document}

\begin{abstract}
	Friedlander and Iwaniec proved that the number of points of the orbit $\{\gamma i \colon \gamma\in\SL_2(\Z)\}$ lying at a distance $p-2$ from the origin $i$ of the upper half-plane, with $p\le x$ prime, is of order $x/\log x$; the upper bound is unconditional, while the lower one rests on a strong hypothesis concerning the distribution of primes in arithmetic progressions. We consider instead square-free distances and prove unconditionally an asymptotic formula, whose main term is of order $x$. Employing the weighted linear sieve, we also show unconditionally that the square-free distances $n$ with at most $7$ prime factors contribute $\gg x/\log x$. Finally, assuming that the sequence $r(n-2)r(n+2)$ has level of distribution $x^{\theta}$ in arithmetic progressions, where $r(n)$ is the number of ways to write $n$ as a sum of two squares, we obtain the same lower bound for the distances with at most $N$ prime factors (the value $7$ corresponds to $\theta=1/6$).
\end{abstract}

\maketitle

\tableofcontents

\section{Introduction}

The prime number theorem states that as $x\to\infty$ we have
\[
    \pi(x) = \sum_{p\le x}1 \sim \int_2^x\frac{dt}{\log t}.
\]

In this paper we consider an analogous problem for the hyperbolic plane. 
Let $\Hh=\{z=x+iy:\ x\in\R,\ y\in\R^+\}$ be the upper half-plane, acted on by the modular group
\[
    \Gamma=\SL_2(\Z)=\left\{\begin{pmatrix}x_1&x_2\\ x_3&x_4\end{pmatrix}:\ x_1,x_2,x_3,x_4\in\Z,\ x_1x_4-x_2x_3=1\right\}
\]

by fractional linear transformations. As a distance function on $\Hh$ we take
\[
    u(z,w)=\frac{|z-w|^2}{\operatorname{Im}z\operatorname{Im}w},
\]

and we choose $z=i$ as the origin, the orbit $\{\gamma i:\gamma\in\Gamma\}$ plays the role of the integers. 
Similarly to \cite[(1.8)]{MR2486486} a direct computation gives
\[
    u(\gamma i,i)+2=x_1^2+x_2^2+x_3^2+x_4^2\qquad\text{for }\gamma=\begin{pmatrix}x_1&x_2\\ x_3&x_4\end{pmatrix}\in\Gamma.
\]

Following \cite{MR2486486}, for a positive integer $n$ we say that $\gamma i$ is at distance $n-2$ from the origin if $x_1^2+x_2^2+x_3^2+x_4^2=n$.

In \cite{MR2486486} Friedlander and Iwaniec studied the number $\pi_\Gamma(x)$ of orbit points at a prime distance $p-2$, $p\le x$. 
They showed that $\pi_\Gamma(x)\asymp x/\log x$ (here $f\asymp g$ means that $c_1 g \le f \le c_2 g$ for some positive absolute constants $c_1$, $c_2$) under the following hypothesis on primes in arithmetic progressions, which interpolates between the Bombieri--Vinogradov theorem and the Elliott--Halberstam conjecture.

\begin{conjecture}[Assumption $A(\theta)$]
    \label{conj}
    Let $E(x,Q)$ denote the remainder term of level $Q$ in Bombieri-Vinogradov theorem that is
	\[
        \begin{split}
            E(x, Q) &= \sum_{q \leq Q} \max_{(a,q)=1} \max_{y \leq x} \left|\psi(y;q,a)-\frac{y}{\varphi(q)}\right| \\
            &=  \sum_{q \leq Q} \max_{(a,q)=1} \max_{y \leq x} |E(y;q,a)|,
        \end{split}
	\]
	where $\varphi(q)$ is the Euler's totient function, $\psi(x;q,a)$ is the Chebyshev function given by
	\[
	    \psi(x;q,a) = \sum_{\substack{n \leq x \\ n \equiv a \pmod q}} \Lambda(n)
	\]
    and $\Lambda(n)$ is the von Mangoldt function. 
    Let also $0 < \theta \leq 1$. 
    Then the bound
    \begin{equation} \tag{$A(\theta)$}
        E(x,Q) \ll \frac{x}{(\log x)^A}
    \end{equation}
    holds for $Q=x^{\theta-\varepsilon}$, for every $\varepsilon >0$ and $A>0$ with an implied constant depending only on $\varepsilon$ and $A$.
\end{conjecture}

The case $\theta=1/2$ is the Bombieri--Vinogradov theorem and the case $\theta=1$ is the Elliott--Halberstam conjecture. 
By \cite{MR2486486} we have $\pi_\Gamma(x)\asymp x/\log x$ provided that $A(\theta)$ holds for some $\theta < 1$ sufficiently close to $1$ (the upper bound $\pi_\Gamma(x)\ll x/\log x$ being unconditional, hence $A(\theta)$ is essential only for the lower bound). 
A more general problem, with a polynomial in place of the quadratic form, was considered by Bourgain, Gamburd and Sarnak in \cite{MR2246331}.

In this paper we replace primes by square-free numbers. 
Let $s_\Gamma(x)$ be the number of points $\gamma i$ at a distance $n-2$ from the origin with $n \leq x$ and $n$ being free of squares, i.e. $s_{\Gamma}(x)$ is the number of quadruples $(x_1,x_2,x_3,x_4) \in \Z^4$ such that $x_1^2+x_2^2+x_3^2+x_4^2 = n \leq x$, $x_1x_4-x_2x_3=1$ and $\mu^2(n)=1$.
Here $\mu$ is the M\"obius function, so that $\mu^2(n)=1$ means that $n$ is square-free. Throughout the paper, $\chi_4$ denotes the non-principal character modulo $4$ and $r(m)$ denotes the number of representations of $m$ as a sum of two squares of integers. 
Our main unconditional result is the following.

\begin{theorem}\label{sqfree-theorem}
    For $x \ge 2$ and any $\varepsilon > 0$ we have
    \[
        s_{\Gamma}(x) = 6Cx + O(x^{11/12+ \varepsilon}),
    \]
    where
    \[
        C = \prod_{p > 2}\biggl(1-\frac{p-\chi_4(p)}{p^2(p+\chi_4(p))}\biggr), \quad 0 < C < 1.
    \]
\end{theorem}

We also consider distances with few prime factors. For an integer $k \ge 1$ let $s_{\Gamma}^{(k)}(x)$ count the number of $(x_1,x_2,x_3,x_4)$ subject to the conditions of $s_{\Gamma}(x)$ such that $\omega(n)\le k$, where $\omega(n)$ is the number of distinct prime factors of $n$.

\begin{theorem}\label{almostprimes-theorem}
    For $x > 3$ we have
    \[
        s_{\Gamma}^{(7)}(x) \gg \frac{x}{\log x}.
    \]
\end{theorem}
The number $7$ comes from the level of distribution $x^{1/6-\varepsilon}$ of the sequence $r(n-2)r(n+2)$ in arithmetic progressions which is available unconditionally in Lemma \ref{convr}.
We formulate the corresponding hypothesis with a general exponent.

\begin{conjecture}[Hypothesis $R(\theta)$] \label{conj-R}
    Let $0 < \theta \le 1$. 
    For every $\varepsilon>0$ and $A>0$ we have
    \[
        \sum_{\substack{d\le x^{\theta-\varepsilon}\\ d\ \mathrm{odd\ square\text{-}free}}}
        \Bigg|\sum_{\substack{n\le x\\ n\equiv3\ \pmod 4\\ n\equiv0\ \pmod d}}r(n-2)r(n+2)-8g(d)x\Bigg|\ll\frac{x}{(\log x)^A},
    \]
    where
    \[
        g(d) = \frac{1}{d} \prod_{p \mid d} \frac{p - \chi_4(p)}{p+\chi_4(p)},
    \]
    see also Lemma \ref{convr}, and an implied constant depending only on $\varepsilon$ and $A$.
\end{conjecture}
By Lemma \ref{convr}, $R(\theta)$ holds for $\theta = 1/6$. 
In analogy with the Elliott--Halberstam conjecture one may expect $R(\theta)$ to hold for every $\theta<1$. 
Note that $R(\theta)$ concerns the weights $r(n-2)r(n+2)$ only and does not involve primes.
Our conditional result is the following.
\begin{theorem}\label{almostprimes-theorem-conditional}
    Let $0 < \theta \le 1$ and assume $R(\theta)$.
    Then, with $N = \lfloor2/\theta\rfloor$, for $x > 3$ we have
    \[
        s_{\Gamma}^{(N)}(x)\gg\frac{x}{\log x}.
    \]
    In particular, $R(1/2)$ implies $s_{\Gamma}^{(4)}(x) \gg x/\log x$.
\end{theorem}

Theorem \ref{almostprimes-theorem-conditional} is obtained from the linear sieve, see Lemma \ref{linear-sieve-lem}, applied to the sequence $r(n-2)r(n+2)$, whose level of distribution, by Lemma~\ref{convr} below, is $x^{1/6-\varepsilon}$. Taking $\theta = 1/6$ in the above theorem gives $\omega(n) \le 12$ unconditionally.
Theorem \ref{almostprimes-theorem} is obtained by adding to this argument the logarithmic weights of Richert \cite{MR318083}, which brings the number $12$ down to $7$.

In a recent paper \cite{shp-xiao}, Shparlinski and Xiao obtain, by rather different methods, an asymptotic formula for the number $S_{\mathrm{sq}}(T)$ of
matrices
\[
    \gamma = \begin{pmatrix}x_1&x_2\\ x_3&x_4\end{pmatrix} \in \SL_2(\Z),
    \qquad \max_{1 \le i \le 4}|x_i| \le T,
\]
for which $n = x_1^2+x_2^2+x_3^2+x_4^2$ is square-free. Their treatment is
considerably more general than ours and proceeds by a $\delta$-method combined
with the weighted linear sieve.

The two problems differ in the way the matrices are ordered: in \cite{shp-xiao}
they are ordered by the sup-norm of the entries, whereas we order them by the value of the quadratic form, that is via $n \le x$. This ordering connects the problem to the sequence $r(n-2)r(n+2)$, so that we may follow the approach of \cite{MR2486486}, allowing us to get a better level of distribution and, therefore, fewer prime factors via Lemma \ref{convr}.

\section{Auxiliary statements}

\subsection{Representations as a sum of two squares}

Let us first recall some known results about $r(n)$. 
In what follows $\chi_4$ is the non-principal Dirichlet character modulo $4$ that is completely multiplicative function given by
\[
	\chi_4(p) = \begin{cases}
					1, &p \equiv 1 \pmod 4,\\
					-1, &p \equiv 3 \pmod 4,\\
					0, &p=2.
				\end{cases}
\]

We use the classical formula
\[
    r(m) = 4\sum_{d \mid m}\chi_4(d), \qquad m \ge 1,
\]
which in particular gives the trivial bound $r(m) \le 4 \tau(m) \ll  m^{\varepsilon}$, where $\tau$ is the divisor function.

The following two results are Lemma 4.1 and Lemma 4.2 of \cite{MR2486486}, respectively, fore references see also results of Smith \cite{MR232741}, \cite{MR249377}.

\begin{lemma}\label{rinAps}
	Let $(a,q)=1$. 
	Then for $q \not\equiv 0 \pmod 4$ we have
	\[
		\sum_{\substack{n \leq x \\ n \equiv a \pmod q}} r(n) = \frac{\pi x}{q}\prod_{p \mid q} \left(1-\frac{\chi_4(p)}{p}\right) + O (q^{-1/2}x^{2/3+\varepsilon})
	\]		
and the main term is multiplied by $1+\chi_4(a)$ for $q \equiv 0 \pmod 4$.
\end{lemma} 

\begin{lemma} \label{convr}
    Let $d$ be an odd positive integer and let $\varepsilon > 0$. Then
    \[
        \sum_{\substack{n\le x\\ n\equiv3\ (\mathrm{mod}\ 4)\\ n\equiv0\ (\mathrm{mod}\ d)}}r(n-2)r(n+2) = 8g(d)x+O\big(d^{-1/2}x^{11/12+\varepsilon}+x^{1/2+\varepsilon}\big),
    \]
    where
    \[
        g(d) = \frac{1}{d} \prod_{p\mid d}\frac{p-\chi_4(p)}{p+\chi_4(p)}
    \]
    and the implied constant depends only on $\varepsilon$. 
    In particular the error term is $O(x^{11/12+\varepsilon})$ uniformly in $d$.
\end{lemma}

Lemma 4.2 of \cite{MR2486486} is stated with the error term $O(x^{11/12+\varepsilon})$, however we need the refined form of Lemma~\ref{convr}, which the proof in \cite{MR2486486} in fact
gives. 
The lemma is stated in \cite{MR2486486} for all odd $d$, not only for square-free $d$, and we shall apply it with $d$ a perfect square. 
Below we recall the argument, indicating the changes in the error term in comparison with \cite{MR2486486}. 

\begin{proof}
    Recall that for $m \equiv 1 \pmod 4$ we have
    \[
    	r(m) = 8 \sum_{\substack{k \mid m \\ k < \sqrt{m}}} \chi_4(k)+4\chi_4(\sqrt{m}),
    \]
    where the last term vanishes if $m$ is not a square.
    Since $n\equiv 3 \pmod 4$ we have $n-2 \equiv 1  \pmod 4$, so
    \[
        r(n-2) = 8\sum_{ \substack{k\mid n-2 \\ k<\sqrt{n-2}}}\chi_4(k)+4\chi_4(\sqrt{n-2}),
    \]
    the last term occurring only when $n-2$ is a square.
    When summed over $n \le x$ the last term contributes
    \[
        \sum_{\substack{n\le x\\ n \equiv 3 \pmod 4 \\ n \equiv 0 \pmod d}}r(n+2) \chi_4 (\sqrt{n-2}) \ll \sum_{m \le x} m^{\varepsilon} \chi_4(\sqrt{m}) \ll x^{1/2+\varepsilon},
    \]
    which is negligible for our goals.

    Writing $k\mid n-2$ as $n\equiv2\pmod k$ we therefore obtain
    \[
        \sum_{\substack{n\le x\\ n \equiv 3 \pmod 4 \\ n \equiv 0 \pmod d}}r(n-2)r(n+2)
        = 8 \sum_{\substack{k < \sqrt {x-2}\\ (k,d) = 1}} \chi_4(k)\sum_{\substack{k^2+2 < n \le x\\ n \equiv 3 \pmod 4 \\ n \equiv 0 \pmod d \\ n \equiv 2 \pmod k}} r(n+2) + O(x^{1/2+\varepsilon}),
    \]
    where only odd $k$ occur, since $\chi_4(k) = 0$ otherwise.
    
    Note that as $n \equiv 0 \pmod d$ and $n \equiv 2 \pmod k$ we have $(k,d)=1$.
    Indeed, assume $(k,d) = \delta > 1$ and write $d = d'\delta$, $k = k'\delta$.
    From $n \equiv 0 \pmod d$ and $n \equiv 2 \pmod k$ we get $n = ud'\delta$ and $n = vk'\delta+2$ for
    some integers $u,v$, thus $\delta(ud'-vk') = 2$ and so $\delta \mid 2$. But $\delta \mid d$ and
    $d$ is odd, therefore $\delta = 1$.
    
    For $(k,d)=1$ and $k$ odd, $n+2$ runs over a residue class modulo $4dk$ which is coprime to $4dk$ and congruent to $1$ modulo $4$, therefore we can apply Lemma \ref{rinAps} and get    
    \[
        \begin{split}
            \sum_{\substack{k^2+2 < n \le x\\ n \equiv 3 \pmod 4 \\ n \equiv 0 \pmod d \\ n \equiv 2 \pmod k}}& r(n+2) = \sum_{\substack{k^2+4 < m \le x+2 \\ m \equiv a \pmod {4dk}}} r(m) \\
            &= \frac{\pi (x-k^2)}{2 dk} h(d) h(k) + O\bigl((dk)^{-1/2} (x^{2/3+\varepsilon}+k^{4/3 + \varepsilon}) + (dk)^{-1+\varepsilon'}\bigr),
        \end{split}
    \]
    where we used
    \[
        \prod_{p \mid 4dk} \biggl( 1-\frac{\chi_4(p)}{p}\biggr) =  h(d)h(k)
    \]
    where $h(m)=\prod_{p\mid m}(1-\chi_4(p)/p)$, and we used that $\chi_4(2)=0$ together with
    $(k,d)=1$. Here $a$ denotes the class given by the Chinese remainder theorem, and the main
    term is multiplied by $1 + \chi_4(a) = 2$, since $4dk \equiv 0 \pmod 4$ and $a \equiv 1 \pmod 4$.
    The term $O\bigl((dk)^{-1+\varepsilon'}\bigr)$ comes from the length of the interval being $x-k^2-2$.

    As $k \ll \sqrt{x}$ the error term contribution is
    \[
        d^{-1/2}\sum_{k \ll \sqrt{x}} \bigl( x^{2/3+\varepsilon} k^{-1/2} + k^{5/6 +\varepsilon} + d^{-1/2+\varepsilon'}k^{-1+\varepsilon'} \bigr) \ll d^{-1/2} x^{11/12+\varepsilon}.
    \]
    Thus we have achieved so far
    \[
        \begin{split}
            \sum_{\substack{ n \le x\\ n \equiv 3 \pmod 4\\ n \equiv 0 \pmod d}} r(n-2)r(n+2) &= 4 \pi \frac{h(d)}{d} \sum_{\substack{k < \sqrt{x}\\ (k,d)=1}} \frac{\chi_4(k)}{k} h(k) (x-k^2) \\
            &+ O(d^{-1/2} x^{11/12+\varepsilon} + x^{1/2+\varepsilon}).
        \end{split}
    \]
    
    Completing the sum over $k$ to infinity and using that $h(p) = 1-\chi_4(p)/p$ we get
    \[
        \begin{split}
            \sum_{(k,d) = 1}\frac{\chi_4(k)h(k)}{k} &= \prod_{p \nmid d} \biggl( 1 + \biggl(1-\frac{\chi_4(p)}{p}\biggr) \sum_{j \ge 1} \frac{\chi_4(p)^j}{p^j} \biggr) \\
            &= \prod_{p \nmid d}\biggl( 1 + \frac{\chi_4(p)}{p}\biggr) = \frac{2}{\pi} \prod_{p\mid d}\biggl(1+\frac{\chi_4(p)}{p}\biggr)^{-1},
        \end{split}
    \]
    where we used that
    \[
        \begin{split}
             \prod_{p}\biggl( 1 + \frac{\chi_4(p)}{p}\biggr) &= \prod_{p}\biggl( 1 - \frac{\chi_4(p)}{p}\biggr)^{-1} \biggl( 1 - \frac{\chi_4^2(p)}{p^2}\biggr) \\ 
             & = \prod_{p}\biggl( 1 - \frac{\chi_4(p)}{p}\biggr)^{-1} \prod_{p > 2}\biggl( 1 - \frac{1}{p^2}\biggr) \\
             &= \frac{L(1, \chi_4)}{\zeta(2)} \biggl(1-\frac{1}{4}\biggr)^{-1} = \frac{2}{\pi}.
        \end{split}
    \]
    Therefore, the contribution of this completed to infinity sum is $x$ times
    \[
        \begin{split}
            \frac{8h(d)}{d} \prod_{p \mid d} \biggl(1+\frac{\chi_4(p)}{p}\biggr)^{-1} &= \frac{8}{d} \prod_{p \mid d} \biggl(1-\frac{\chi_4(p)}{p}\biggr) \biggl(1+\frac{\chi_4(p)}{p}\biggr)^{-1}\\
            &= \frac{8}{d} \prod_{p \mid d} \frac{p-\chi_4(p)}{p+\chi_4(p)} = 8g(d),
        \end{split}
    \]
    which is our expected main term.
    We note that $0 \le g(p) < 1$ for every odd prime $p > 2$, and that
    \[
        g(p) = \frac{1}{p}-\frac{2\chi_4(p)}{p(p+\chi_4(p))} = \frac{1}{p} + O\biggl(\frac{1}{p^2}\biggr), \quad
        g(p^2) = \frac{p-\chi_4(p)}{p^2(p+\chi_4(p))} \le \frac{2}{p^2}.
    \]
    It remains to show that the tail and terms coming from $k^2$ are small and hence go to the error term.
    Notice that elementarily we have
    \[
        \sum_{\substack{k \le t \\ (k,d)=1}} \chi_4(k) = \sum_{d' \mid d} \mu(d') \chi_4(d')  \sum_{m \le t/d'} \chi_4(m) \ll \tau(d).
    \]

    For $p \nmid 2d$ and $\text{Re} s > 1$ we have
    \[
        \sum_{j \ge 0} \frac{\chi_4(p^j)h(p^j)}{p^{js}}
        = 1 + \Bigl(1-\frac{\chi_4(p)}{p}\Bigr)\frac{\chi_4(p)p^{-s}}{1-\chi_4(p)p^{-s}}
        = \frac{1-p^{-s-1}}{1-\chi_4(p)p^{-s}},
    \]
    so that, comparing coefficients in
    \[
        \sum_{(k,d) = 1}\frac{\chi_4(k)h(k)}{k^{s}}
     = \prod_{p \nmid d}\biggl(1-\frac{\chi_4(p)}{p^{s}} \biggr)^{-1} \prod_{p \nmid 2d}\biggl(1-\frac{1}{p^{s+1}}\biggr),
    \]
    for every $k$ with $(k,d)=1$ one has
    \[
        \chi_4(k) h(k) = \sum_{k' \mid k} \chi_4(k/k') \frac{\mu(k')}{k'},
    \]
    where $k'$ runs over odd square-free numbers.
    Therefore
    \[
        \sum_{\substack{k \le t \\ (k,d)=1}} \chi_4(k) h(k) = \sum_{\substack{k' \le t \\ (k', 2d)=1}} \frac{\mu(k')}{k'} \sum_{\substack{u \le t/k'\\ (u,d)=1}} \chi_4(u)
    \]
    and consequently
    \[
        \sum_{\substack{k \le t \\ (k,d)=1}} \chi_4(k) h(k) \ll \tau(d) \sum_{k' \le t} \frac{|\mu(k')|}{k'} \ll \tau(d) \log t.
    \]
    Thus, by partial summation we have
    \[
        \begin{split}
            \sum_{\substack{k > \sqrt{x} \\ (k,d)=1}} \frac{\chi_4(k)}{k} h(k) &\ll \frac{\tau(d) \log x}{\sqrt{x}} + \tau(d) \int_{\sqrt{x}}^{\infty} \frac{\log t}{t^2} dt \ll \frac{\tau(d) \log x}{\sqrt{x}}
        \end{split} 
    \]
    and the tail contribution can be bounded as
    \[
        4 \pi x \frac{h(d)}{d} \sum_{\substack{k > \sqrt{x} \\ (k,d)=1}} \frac{\chi_4(k)}{k} h(k) \ll x^{1/2 + \varepsilon}.
    \]
    It remains to bound the contribution of terms involving $k^2$, which up to a constant factor time $h(d)/d$ is
    \[
        \begin{split}
            \sum_{\substack{k \le \sqrt{x} \\ (k,d)=1}} \chi_4(k) h(k) k &= \sqrt{x} \sum_{\substack{k \le \sqrt{x} \\ (k,d)=1}} \chi_4(k) h(k)
            -   \int_1^{\sqrt{x}} \biggl( \sum_{\substack{k \le t \\ (k,d)=1}} \chi_4(k) h(k) \biggr) dt\\
            &\ll \tau(d) \sqrt{x} \log x,
        \end{split} 
    \]
    therefore
    \[
         \frac{h(d)}{d} \sum_{\substack{k \le \sqrt{x} \\ (k,d)=1}} \chi_4(k) h(k) k \ll \frac{h(d)}{d} \tau(d) \sqrt{x} \log x \ll  x^{1/2 + \varepsilon}.
    \]
    Finally, we may conclude
    \[
            \sum_{\substack{ n \le x\\ n \equiv 3 \pmod 4\\ n \equiv 0 \pmod d}} r(n-2)r(n+2) = 8 g(d) x
            + O(d^{-1/2} x^{11/12+\varepsilon} + x^{1/2+\varepsilon})
    \]
    as required.
\end{proof}

In order to estimate the contribution of even $d$ in Theorem \ref{sqfree-theorem} we need the following lemma.

\begin{lemma}\label{convr-even}
    Let $q$ be an odd positive integer, let $a$ be an integer with $(a(a+1),q) = 1$ and let
    $\varepsilon > 0$. Then
    \[
        \sum_{\substack{n \le x\\ n \equiv a \pmod q}} r(n) r(n+1)
        = 8g(q)x + O\bigl(q^{-1/2}x^{11/12+\varepsilon}+x^{1/2+\varepsilon}\bigr),
    \]
    with $g$ as in Lemma \ref{convr} and the implied constant depending only on $\varepsilon$.
\end{lemma}

\begin{proof}
    Exactly one of $n$, $n+1$ is odd, and for odd $m$ one has $r(m) = 0$ unless $m \equiv 1 \pmod 4$. If $n \equiv 2 \pmod 4$ then $n+1 \equiv 3 \pmod 4$, so $r(n+1) = 0$.
    Similarly, if $n \equiv 3 \pmod 4$ then $r(n) = 0$. 
    Hence only $n \equiv 0, 1 \pmod 4$ contribute to the sum.
    We work these cases out separately.

    First, let $n \equiv 0 \pmod 4$ and denote its contribution by $S_1(x)$. As $n + 1\equiv 1 \pmod 4$ we have
    \[
        r(n+1) = 8\sum_{\substack{k \mid n+1\\ k < \sqrt{n + 1}}}\chi_4(k) + 4\chi_4(\sqrt{n+1}),
    \]
    and
    \[
        \begin{split}
            S_1(x) &= 8 \sum_{\substack{n \le x\\ n \equiv a \pmod q \\ n \equiv 0 \pmod 4}} r(n) \sum_{\substack{k \mid n+1\\ k < \sqrt{n + 1}}}\chi_4(k) + O(x^{1/2 + \varepsilon})\\
            &= 8 \sum_{k < \sqrt{x+1}} \chi_4(k) \sum_{\substack{n \le x \\ n \equiv 0 \pmod 4\\ n \equiv a \pmod q \\ n \equiv -1 \pmod k\\ n > k^2-1}} r(n) + O(x^{1/2 + \varepsilon}).
        \end{split}
    \]
    Here $k$ is odd and $(k,q) = 1$ --- otherwise the inner sum is empty, since a prime $p \mid (k,q)$ would divide $n+1$ and $q$, hence $a+1$, contrary to hypothesis. Writing $k \mid n+1$ as $n \equiv -1 \pmod k$ and substituting $n = 4m$, so that $r(n) = r(m)$, the conditions on $m$ become a single class $a' \pmod{qk}$ coprime to $qk$, and $(k^2-1)/4 < m \le x/4$ (note that $(k^2-1)/4$ is an integer as $k$ is odd). Since $qk$ is odd, Lemma \ref{rinAps} applies
    \[
        \begin{split}
            \sum_{\substack{n \le x \\ n \equiv 0 \pmod 4\\ n \equiv a \pmod q \\
                n \equiv -1 \pmod k\\ n > k^2-1}} r(n)
            &= \sum_{\substack{m \le x/4 \\ m \equiv a' \pmod {qk} \\
                m > (k^2-1)/4}} r(m)\\
            &= \frac{\pi (x-k^2)}{4qk}h(q)h(k)
               + O\bigl((qk)^{-1/2}(x^{2/3+\varepsilon}+k^{4/3+\varepsilon})\bigr),
        \end{split}
    \]
    the $+1$ left over from $x-k^2+1$ being absorbed into the error term.
    
    Summing the error over $k$ we obtain
    \[
        \begin{split}
            \sum_{k < \sqrt{x+1}}(qk)^{-1/2}\bigl(x^{2/3+\varepsilon}+k^{4/3+\varepsilon}\bigr)
            &\ll q^{-1/2}x^{2/3+\varepsilon}\sum_{k \ll \sqrt{x}} k^{-1/2} \\
            &\ll q^{-1/2} x^{11/12+\varepsilon}.
        \end{split}
    \]
    In the main term $\chi_4$ vanishes at even $k$, thus
    \[
        8\sum_{\substack{k<\sqrt{x+1}\\ (k,q)=1}}\chi_4(k)\frac{\pi(x-k^2)}{4qk}h(q)h(k)
        = \frac{2\pi h(q)}{q}\sum_{\substack{k<\sqrt{x+1}\\ (k,q)=1}}
          \chi_4(k)h(k)\Bigl(\frac{x}{k}-k\Bigr).
    \]
    As in the proof of Lemma \ref{convr} we have
    \[
        \sum_{\substack{k \le t \\ (k,q)=1}} \chi_4(k)h(k) \ll \tau(q)\log t,
    \]
    whence by partial summation
    \[
        \sum_{\substack{k<\sqrt{x+1}\\ (k,q)=1}}\chi_4(k)h(k)k \ll \tau(q)x^{1/2}\log x,
        \quad
        \sum_{\substack{k \ge \sqrt{x+1}\\ (k,q)=1}}\frac{\chi_4(k)h(k)}{k}
        \ll \frac{\tau(q)\log x}{x^{1/2}}.
    \]
    Since $h(q)\tau(q) \ll q^{\varepsilon}$, both contribute $O(x^{1/2+\varepsilon})$
    after multiplication by $2\pi h(q)/q$ and $2\pi h(q)x/q$ respectively. On completing the sum to infinity we get
    \[
        \sum_{\substack{k \ge 1\\ (k,q)=1}}\frac{\chi_4(k)h(k)}{k}
        = \prod_{p \nmid q}\Bigl(1+\frac{\chi_4(p)}{p}\Bigr)
        = \frac{2}{\pi}\prod_{p \mid q}\Bigl(1+\frac{\chi_4(p)}{p}\Bigr)^{-1},
    \]
    so that the main term equals
    \[
        \frac{2\pi h(q)x}{q}\cdot\frac{2}{\pi}
        \prod_{p \mid q}\Bigl(1+\frac{\chi_4(p)}{p}\Bigr)^{-1}
        = \frac{4x}{q}\prod_{p \mid q}\frac{p-\chi_4(p)}{p+\chi_4(p)} = 4g(q)x.
    \]
    Therefore
    \[
        S_1(x) = 4g(q)x + O\bigl(q^{-1/2}x^{11/12+\varepsilon}+x^{1/2+\varepsilon}\bigr).
    \]
    
    Next, let $n \equiv 1 \pmod 4$ and denote its contribution by $S_2(x)$. This
    time $n$ itself is odd and $n \equiv 1 \pmod 4$, so
    \[
        r(n) = 8\sum_{\substack{k \mid n\\ k < \sqrt{n}}}\chi_4(k) + 4\chi_4(\sqrt{n}),
    \]
    therefore
    \[
        S_2(x) = 8\sum_{k<\sqrt{x}}\chi_4(k)
        \sum_{\substack{n \le x \\ n \equiv 1 \pmod 4\\ n \equiv a \pmod q \\
            n \equiv 0 \pmod k\\ n > k^2}} r(n+1) + O(x^{1/2+\varepsilon}).
    \]
    Again $k$ is odd and $(k,q)=1$, since a prime $p \mid (k,q)$ would divide $n$
    and $q$, hence $a$. Substituting $n+1 = 2m$, so that $r(n+1)=r(m)$, the
    condition $n \equiv 1 \pmod 4$ makes $m$ odd, the conditions on $m$ become a
    single class $a'' \pmod{2qk}$ coprime to $2qk$, and
    $(k^2+1)/2 < m \le (x+1)/2$. As $\chi_4(2)=0$ we have $h(2qk)=h(q)h(k)$, and
    since $4 \nmid 2qk$ Lemma \ref{rinAps} gives
    \[
        \begin{split}
            \sum_{\substack{m \le (x+1)/2 \\ m \equiv a'' \pmod{2qk} \\
                m > (k^2+1)/2}}& r(m)
            = \frac{\pi}{2qk}\Bigl(\frac{x+1}{2}-\frac{k^2+1}{2}\Bigr)h(q)h(k) \\
               &+ O\bigl((qk)^{-1/2}(x^{2/3+\varepsilon}+k^{4/3+\varepsilon})\bigr)\\
            &= \frac{\pi (x-k^2)}{4qk}h(q)h(k)
               + O\bigl((qk)^{-1/2}(x^{2/3+\varepsilon}+k^{4/3+\varepsilon})\bigr).
        \end{split}
    \]
    Proceeding as in the case of $S_1(x)$ we conclude
    \[
        S_2(x) = 4g(q)x + O\bigl(q^{-1/2}x^{11/12+\varepsilon}+x^{1/2+\varepsilon}\bigr).
    \]
    Adding $S_1(x)$ and $S_2(x)$ completes the proof.
\end{proof}

\subsection{The linear sieve} \label{linear-sieve-section}

Let $\mathcal A = (a_n)_{n \le x}$ be a finite sequence of non-negative reals, let $\mathcal P$ be a set of primes, $P(z) = \prod_{p < z,\, p \in \mathcal P}p$, and
\[
    S(\mathcal A, \mathcal P, z) = \sum_{\substack{n\le x\\ (n,P(z)) = 1}}a_n, \quad
    \mathcal A_d = \sum_{\substack{n\le x\\ n \equiv 0 \pmod  d}} a_n .
\]
Suppose that there are $X > 0$, a multiplicative function $g$ with $0 \le g(p) < 1$ for $p \in \mathcal P$, and real numbers $r_d$ such that
\[
    \mathcal A_d = g(d)X+r_d
\]
for all square-free $d \mid P(z)$, and that $g$ satisfies the linear sieve condition
\[
    \prod_{\substack{w\le p<z\\ p\in\mathcal P}}(1-g(p))^{-1}\le\frac{\log z}{\log w}\left(1+\frac{L}{\log w}\right)
\]
for $2 \le w \le z$ and some constant $L \ge 1$. 
Put 
\[
    V(z) = \prod_{\substack{p < z \\ p \in \mathcal P}} \bigl(1-g(p)\bigr).
\]

\begin{lemma} \label{linear-sieve-lem}
    Let $D \ge z \ge 2$ and $s = \log D/\log z$. 
    For $2 \le s \le 4$ define
    \[
        f(s)=\frac{2e^{\gamma}\log(s-1)}{s}.
    \]
    Similarly, for $1 \le s \le 3$ set $F(s) = 2e^{\gamma}/{s}$.
    Under the above hypotheses we have
    \[
        \begin{split}
            S(\mathcal A,\mathcal P,z) &\ge XV(z)\big(f(s) - B(\log D)^{-1/14}\big) - \sum_{\substack{d < D\\ d\mid P(z)}}|r_d|,\\
            S(\mathcal A,\mathcal P,z) &\le XV(z)\big(F(s) + B(\log D)^{-1/14}\big) + \sum_{\substack{d < D\\ d\mid P(z)}}|r_d|,
        \end{split}
    \]
    where $B$ depends only on $L$.
\end{lemma}

This is Theorem~8.4 of \cite{MR424730} (see also \cite{MR202680} and \cite[Theorem 9.7]{MR1477155}).

Finally we recall Mertens' formula in the weak form
\[
    \prod_{2<p<z}\left(1-\frac1p\right)\asymp\frac {1}{\log z}
\]
which gives, for $g$ as in Lemma \ref{convr},
\[
    V(z) = \prod_{2 < p < z}(1-g(p)) \asymp \frac{1}{\log z},
\]
since $\prod_{p>2}(1-g(p))(1-1/p)^{-1}$ converges absolutely to a positive limit.
\section{Reduction to sums of two squares}\label{initial-reduction}

Recall that $s_{\Gamma}(x)$ is the number of quadruples $(x_1,x_2,x_3,x_4) \in \Z^4$ subject to $n = x_1^2 + x_2^2 + x_3^2 + x_4^2 \le x$, $x_1x_4-x_2x_3 = 1$ and $\mu^2(n) = 1$.

Changing variables as in \cite{MR2486486}, we have
\[
    y_1 = x_1 + x_4, \quad 
    y_2 = x_2 - x_3, \quad 
    y_3 = x_1 - x_4, \quad 
    y_4 = x_2 + x_3
\]
or, equivalently,
\[
    x_1 = \frac{y_1 + y_3}{2}, \quad 
    x_2 = \frac{y_2 + y_4}{2}, \quad 
    x_3 = \frac{y_4 - y_2}{2}, \quad 
    x_4 = \frac{y_1 - y_3}{2}.
\]

We readily see that under this change of variables in $s_\Gamma$ we have
\[
    y_1^2+y_2^2=n+2, \quad (y_1^2+y_2^2)-(y_3^2+y_4^2)=4, \quad \mu^2(n)=1.
\]

Note that if $n$ is an odd integer, then $x_i \in \Z$ implies $y_i \in \Z$. 
Conversely, given $y_i$ we have that $y_1$ has the opposite parity from $y_2$, same holds for $y_3$ and $y_4$.
Then given $y_i$ precisely one of the quadruples $(y_1,y_2,y_3,y_4)$ and $(y_1,y_2,y_4,y_3)$ gives rise to $x_i$.
Thus the number of integer solutions in $x$'s is half of the number of integer solutions in $y$'s.

For even $n$ the two sums $y_1^2 + y_2^2$ and $y_3^2 + y_4^2$ differ by $4$, hence are congruent modulo $4$, so either all four $y_i$ are even or all four are odd --- in both cases the $x_i$ are integers and the correspondence is one-to-one.
\section{Square-free case: proof of Theorem \ref{sqfree-theorem}}

Let $r(n)$ denote the number of representations of $n$ as the sum of two squares.
The number of integer quadruples $(y_1,y_2,y_3,y_4)$  subject to the above conditions equals $r(n+2)r(n-2)$, as the pairs $(y_1, y_2)$ and $(y_3, y_4)$ are chosen independently.
By the above, the odd $n$ contribute half of this number and the even $n$ all of it, so that
\begin{equation} \label{sqfree-decomp}
    s_{\Gamma}(x) = \frac{1}{2}s_{\Gamma}^{\text{odd}}(x)+ s_{\Gamma}^{\text{even}}(x),
\end{equation}
where
\[
    \begin{split}
        s_{\Gamma}^{\text{odd}}(x) &= \sum_{\substack{n\le x\\ n \text{ odd}}}\mu^2(n)r(n-2)r(n+2),
        \\
        s_{\Gamma}^{\text{even}}(x) &= 
        \sum_{\substack{n\le x\\ n \text{ even}}}\mu^2(n)r(n-2)r(n+2).
    \end{split}
\]

Notice that only $n \equiv 3 \pmod 4$ contribute to $s_{\Gamma}^{\text{odd}}(x)$, since $r(n \pm 2) = 0$
for $n \equiv 1 \pmod 4$. An even square-free $n$ is of the form $n = 4m+2$ with $2m+1$
square-free, and then $n-2 = 4m$, $n+2 = 4(m+1)$, so that $r(n-2)r(n+2) = r(m)r(m+1)$ as
$r(4k) = r(k)$, while $\mu^2(n) = \mu^2(2m+1)$. 
Hence
\[
    s_{\Gamma}^{\text{even}}(x) = \sum_{m \le (x-2)/4}\mu^2(2m+1)r(m)r(m+1).
\]

The same decomposition holds for $s_{\Gamma}^{(k)}(x)$ with the additional condition $\omega(n) \le  k$ inserted throughout.

Below we estimate the odd and even contributions separately, see \eqref{sq-free-odd-contrib} and \eqref{sq-free-even-contrib}, which brings us to
\[
    s_{\Gamma}(x) = \biggl(\frac{1}{2} (8C) + 2C\biggr)x + O(x^{11/12+\varepsilon}) = 6C x + O(x^{11/12+\varepsilon}).
\]

\subsection{Odd contribution}

We begin with evaluating $s_{\Gamma}^{\text{odd}}(x)$.
For $n \ge 3$ write $a_n = r(n-2)r(n+2)$, and put
\[
    s_{\Gamma}^{\text{odd}}(x) = \sum_{\substack{n \le x\\ n \equiv 3 \pmod 4}}\mu^2(n)a_n,\quad
    s_{\Gamma, d}^{\text{odd}}(x) = \sum_{\substack{n\le x\\ n \equiv 3 \pmod 4\\ n \equiv 0 \pmod d}} a_n.
\]
Note that $a_n \ge 0$, and that $a_n \ll \tau(n+2)\tau(n-2) \ll x^{\varepsilon}$ for $n \le x$. Since we sum over $n \equiv 3 \pmod 4$, then $s_{\Gamma, d}^{\text{odd}}(x) = 0$ for even $d$, so we can freely impose the condition of $d$ being odd.

By Lemma \ref{convr} we have
\[
    s_{\Gamma, d}^{\text{odd}}(x) = 8g(d)x+O\big(d^{-1/2}x^{11/12+\varepsilon}+x^{1/2+\varepsilon}\big).
\]

Using $\mu^2(n) = \sum_{d^2\mid n}\mu(d)$ we obtain
\[
    s_{\Gamma}^{\text{odd}}(x) = \sum_{\substack{n \le x \\ n \equiv 3 \pmod 4}} a_n \sum_{d^2 \mid n} \mu(d) = \sum_{d \le \sqrt{x}} \mu(d) s_{\Gamma, d^2}^{\text{odd}}(x).
\]
We truncate the sum at $y=x^{1/4}$ and get
\[
    s_{\Gamma}^{\text{odd}}(x) = \sum_{d \le y} \mu(d) s_{\Gamma, d^2}^{\text{odd}}(x) + O \biggl( \sum_{y < d \le \sqrt{x}} s_{\Gamma, d^2}^{\text{odd}}(x) \biggr).
\]
Bounding the summand in the error term trivially via
\[
    s_{\Gamma, d^2}^{\text{odd}}(x) \le x^\varepsilon \sum_{\substack{n \le x \\ n \equiv 0 \pmod {d^2}}} 1 \le x^\varepsilon \biggl(\frac{x}{d^2}+1 \biggr)
\]
we get
\[
    \sum_{y < d \le \sqrt x}s_{\Gamma, d^2}^{\text{odd}}(x) \ll  x^{\varepsilon} \sum_{y < d \le \sqrt x}\left(\frac{x}{d^2}+1\right)\ll x^{\varepsilon}\left(\frac xy+\sqrt x\right)\ll x^{3/4+\varepsilon}.
\]

For the main sum we apply Lemma \ref{convr} with the odd modulus $d^2$, which gives
\[
    \begin{split}
        \sum_{\substack{d \le y\\ d\ \mathrm{odd}}} \mu(d) s_{\Gamma, d^2}^{\text{odd}}(x) &= 8x \sum_{\substack{d\le y\\ d\ \text{odd}}} \mu(d) g(d^2) + O\bigl(x^{11/12+\varepsilon} \log y + x^{1/2+\varepsilon}y\bigr)\\
        &= 8x \sum_{\substack{d\le y\\ d\ \text{odd}}} \mu(d) g(d^2) + O\bigl(x^{11/12+\varepsilon}\bigr).
    \end{split}
\]
Since $g(p^2) \le 2/p^2$ we have
\[
    g(d^2) \le \frac{2^{\omega(d)}}{d^{2}} \le \frac{\tau(d)}{d^2} \ll \frac{1}{d^{2-\varepsilon}},
\]
so the series $\sum_{d\ \mathrm{odd}}\mu(d)g(d^2)$ converges absolutely.
On completing the sum to infinity we get
\[
    \sum_{\substack{d \ge 1\\ d\ \text{odd}}}\mu(d)g(d^2) = \prod_{p > 2} \bigl(1-g(p^2)\bigr) = \prod_{p > 2}\biggl(1-\frac{p-\chi_4(p)}{p^2(p+\chi_4(p))}\biggr)=C,
\]
while the tail bounds as
\[
    \sum_{\substack{d > y \\ d\ \text{odd}}}\mu(d)g(d^2) \ll \sum_{d > y} \frac{1}{d^{2-\varepsilon}} \ll \frac{1}{y^{1-\varepsilon}} \ll x^{-1/4+\varepsilon}.
\]
Therefore we  can conclude
\begin{equation} \label{sq-free-odd-contrib}
    s_{\Gamma}^{\text{odd}}(x) = 8Cx + O(x^{11/12+\varepsilon}).
\end{equation}

\subsection{Even contribution}
We evaluate $s_{\Gamma}^{\mathrm{even}}(x)$
via Lemma \ref{convr-even}.
Our goal is to show that for every $\varepsilon > 0$ we have $s_{\Gamma}^{\text{even}}(x) = 2Cx+O(x^{11/12+\varepsilon})$,
with $C$ as in Theorem \ref{sqfree-theorem}.
First, we have 
\[
    s_{\Gamma}^{\mathrm{even}}(x)=\sum_{n \le X}\mu^2(2n+1)r(n)r(n+1)
\]
with $X=(x-2)/4$. 
Expanding $\mu^2(2n+1) = \sum_{d^2 \mid 2n+1}\mu(d)$, with $d$ odd, the condition
$d^2 \mid 2n+1$ is equivalent to $n \equiv b \pmod{d^2}$ with $b = (d^2-1)/2$,
and this class satisfies $2b \equiv -1$ and $2(b+1) \equiv 1 \pmod{d^2}$, so
that $(b(b+1),d^2) = 1$. Since $d^2 \mid 2n+1 \le 2X+1$, we obtain
\[
    s_{\Gamma}^{\mathrm{even}}(x) = \sum_{d \le \sqrt{2X+1}} \mu(d)
    \sum_{\substack{n \le X\\ n \equiv b \pmod{d^2}}} r(n)r(n+1).
\]
As in the case of the odd contribution take $y = x^{1/4}$. Bounding the inner
sum trivially via
\[
    \sum_{\substack{n \le X\\ n \equiv b \pmod{d^2}}} r(n)r(n+1)
    \le x^{\varepsilon}\sum_{\substack{n \le X\\ n \equiv b \pmod{d^2}}} 1
    \le x^{\varepsilon}\biggl(\frac{X}{d^2}+1\biggr),
\]
we bound the tail by
\[
    \begin{split}
        \sum_{y < d \le \sqrt{2X+1}}\ \sum_{\substack{n \le X\\ n \equiv b \pmod{d^2}}}
        r(n)r(n+1)
        &\ll x^{\varepsilon}\sum_{y < d \le \sqrt{2X+1}}\biggl(\frac{X}{d^2}+1\biggr)\\
        &\ll x^{\varepsilon}\biggl(\frac{X}{y}+\sqrt{x}\biggr) \ll x^{3/4+\varepsilon}.
    \end{split}
\]
Since $d$ is odd, so is $d^2$, and the implied constant in Lemma
\ref{convr-even} depends only on $\varepsilon$, so the estimate it provides is
uniform in $d$, the class $b$ varying with $d$. As $X \ll x$, Lemma
\ref{convr-even} with $q = d^2$ gives
\begin{equation} \label{sq-free-even-contrib}
    \begin{split}
        s_{\Gamma}^{\text{even}}(x) &= 8X\sum_{\substack{d\le y\\ d\ \mathrm{odd}}}\mu(d)g(d^2) + O\bigl(X^{11/12+\varepsilon}+yX^{1/2+\varepsilon}+x^{3/4+\varepsilon}\bigr)\\
        &= 8CX + O(x^{11/12+\varepsilon}) = 2Cx + O(x^{11/12+\varepsilon}),
    \end{split}
\end{equation}
since $\sum_{d\ \mathrm{odd}}\mu(d)g(d^2)=\prod_{p>2}(1-g(p^2))=C$ and the tail $d > y$ contributes $ \ll x y^{-1+\varepsilon}$.

\section{Almost primes: proof of Theorem~\ref{almostprimes-theorem-conditional}}\label{sieve-section}

We apply the linear sieve, i.e. Lemma \ref{linear-sieve-lem} to the sequence
\[
    \mathcal A = (a_n)_{n \le x}, \quad a_n = \begin{cases}
                                                    r(n-2)r(n+2),&n \equiv 3 \pmod 4,\\ 0,&\text{otherwise},
                                                \end{cases}
\]
with $\mathcal P$ the set of odd primes. The hypotheses of Section \ref{linear-sieve-section} hold with
\[
    X = 8x,\quad g(d) = \frac{1}{d} \prod_{p \mid d}\frac{p-\chi_4(p)}{p+\chi_4(p)}
\]
and $r_d = \mathcal A_d-8g(d)x$ for odd square-free $d$.

Note that $0 \le g(p) < 1$, and the linear sieve condition is satisfied with some absolute $L$. 
By definition of $g$ and Mertens' theorem we have
\[
    \log\bigl((1-g(p))^{-1}\bigr) = \frac{1}{p}+O(p^{-2}), \quad \sum_{w \le p < z}\frac{1}{p} = \log\biggl(\frac{\log z}{\log w}\biggr)+O\bigl((\log w)^{-1}\bigr).
\]

Let $0 < \theta \le 1$ be such that $R(\theta)$ holds. By Lemma \ref{convr} we can always take $\theta = 1/6$, since for $D \le x^{1/6-4\varepsilon}$
\begin{equation}\label{eq:level}
    \begin{split}
        \sum_{\substack{d\le D\\ d\ \mathrm{odd\ square\text{-}free}}}|r_d| &\ll \sum_{d\le D}\big(d^{-1/2}x^{11/12+\varepsilon}+x^{1/2+\varepsilon}\big) \\
        &\ll D^{1/2}x^{11/12+\varepsilon}+Dx^{1/2+\varepsilon}\ll x^{1-\varepsilon}.
    \end{split}
\end{equation}
Note that using only the weaker error term $O(x^{11/12+\varepsilon})$, i.e. the one of \cite[Lemma 4.2]{MR2486486} one would get $\theta = 1/12$ here.

Take $N = \lfloor2/\theta \rfloor$. Since $2/\theta < N+1$, we can fix $\varepsilon>0$ and $0 < \eta \le 2$ so that
\[
    \frac{2 + \eta}{\theta - 4\varepsilon} < N+1 .
\]
For $\theta = 1/6$ and $N = 12$ one may take $\varepsilon = 10^{-3}$ and $\eta = 10^{-2}$. Now set
\[
    D = x^{\theta-4\varepsilon}, \quad s = 2+\eta,\quad z = D^{1/s}.
\]
By $R(\theta)$ with $A = 2$, or by \eqref{eq:level} when $\theta=1/6$, we have 
\[
    \sum_{\substack{d<D \\ d\mid P(z)}}|r_d| \ll \frac{x}{\log^2 x}.
\]
Hence Lemma \ref{linear-sieve-lem} together with $V(z) \asymp 1/\log z$ gives
\[
    S(\mathcal A,\mathcal P,z) \ge 8xV(z)\biggl(f(s)-B(\log D)^{-1/14} \biggr)-O\big(x(\log x)^{-2}\big)\gg\frac{x}{\log x}
\]
for all sufficiently large $x$ (here we used that $f(s) = 2e^\gamma\log(1+\eta)/(2+\eta) > 0$).

Every $n$ with $a_n \ne 0$ and $(n, P(z)) = 1$ is odd and has all its prime factors at least $z$. Since $n \le x$, the number $\Omega(n)$ of prime factors of $n$ counted with multiplicity satisfies
\[
    \Omega(n) \le \frac{\log x}{\log z} = \frac{s}{\theta - 4\varepsilon} = \frac{2+\eta}{\theta - 4\varepsilon}< N+1,
\]
by the choice of $\varepsilon$ and $\eta$, so that $\Omega(n)\le N$.

It remains to discard the $n$ that are not square-free. Such an $n$ counted in $S(\mathcal A,\mathcal P,z)$ is divisible by $p^2$ for some prime $p$ with $z\le p\le\sqrt x$, and by the trivial bound for $r$
\[
    \sum_{z \le p \le \sqrt x}\ \sum_{\substack{n \le x\\ p^2 \mid n}}a_n\ll x^{\varepsilon} \sum_{z \le p  \le \sqrt{x} }\biggl(\frac{x}{p^2}+1\biggr)\ll x^{\varepsilon}\biggl(\frac xz+\sqrt x\biggr) = o\biggl(\frac{x}{\log x}\biggr),
\]
because $z$ is a fixed positive power of $x$. Hence
\[
    \sum_{\substack{n\le x \\ n \equiv 3 \pmod 4\\ \mu^2(n) = 1 \\ \omega(n) \le N}} r(n-2)r(n+2) \ge S(\mathcal A,\mathcal P,z)-o\biggl(\frac x{\log x}\biggr) \gg \frac{x}{\log x},
\]
where we used that $\omega(n) = \Omega(n)$ for square-free $n$. By \eqref{sqfree-decomp}, with the condition $\omega(n) \le N$ inserted, the left-hand side is at most $2s_{\Gamma}^{(N)}(x)$, and the lower bound $s_{\Gamma}^{(N)}(x) \gg x/\log x$ follows for large $x$ (for $x$ bounded it is trivial since $s_{\Gamma}^{(N)}(x)\ge s_{\Gamma}^{(N)}(3) > 0$). 
This finishes the proof of Theorem \ref{almostprimes-theorem-conditional}. Taking $\theta = 1/6$ and $N = 12$ gives $s_{\Gamma}^{(12)}(x)\gg x/\log x$ unconditionally.

\begin{remark}
   The exponent $1/6$ in the unconditional level of distribution comes from the individual error term $O(d^{-1/2}x^{11/12+\varepsilon})$ in Lemma \ref{convr}, which in turn comes from the summation of terms coming from Lemma \ref{rinAps} over the moduli $4dk$ with $k< \sqrt x$. The value $7$ is what the level $1/6$ allows. By Theorem \ref{almostprimes-theorem-conditional} the bound $\omega(n) \le k$ requires the stronger condition
    \[
        \theta > \biggl(k+1-\frac{\log 4}{\log 3}\biggr)^{-1},
    \]
    which $\theta = 1/6$ satisfies for $k = 7$ and fails for $k = 6$. Thus the number $7$ is not improved by any rearrangement of the sieve, only by a larger level of distribution.
\end{remark}
\section{The weighted sieve: proof of Theorem \ref{almostprimes-theorem}}\label{weighted-section}

We keep the notation of Section \ref{sieve-section}, see
\cite[Chapter 5]{MR1836967} for the general theory of the weighted sieve.
Set $\mathcal{A} = (a_n)$ with $a_n = r(n-2)r(n+2)$ for $n \le x$, $n \equiv 3\pmod 4$ and $a_n = 0$ otherwise, $\mathcal P$ the odd primes, $X = 8x$, $g$ as in Lemma \ref{convr}, $r_d = \mathcal{A}_d - 8g(d)x$, and 
\[
    V(z) = \prod_{2 < p < z}\bigl(1-g(p)\bigr) \asymp \frac{1}{\log z}.
\]
Fix $\varepsilon > 0$ and put
\[
    \alpha = 1/6 - 4\varepsilon, \quad D = x^{\alpha}, \quad z = D^{1/4} = x^{\alpha/4},\quad y = x^{3/20},\quad \lambda = 4/3.
\]
Taking $\varepsilon$ small enough so that $\alpha-3/20 \ge 1/100$ (i.e. $\varepsilon \le 1/600)$ we have $z < y < D$ and $D/y = x^{\alpha-3/20}\ge x^{1/100}$. By \eqref{eq:level}, the remainders satisfy $\sum_{m\le D}|r_m|\ll x^{1-\varepsilon}$, where the sum is taken over odd square-free $m$.

Consider the weighted sum
\[
    W(x) = \sum_{\substack{n\le x\\ (n,P(z))=1}} a_n b_n, \quad\text{where }
b_n=1-\frac1\lambda\sum_{\substack{p\mid n\\ z\le p<y}}\Big(1-\frac{\log p}{\log y}\Big).
\]

Let us begin with interpreting what $b_n >0$ implies for us.
Let $n$ be square-free with $(n,P(z)) = 1$ and $b_n > 0$. All prime factors of $n$ are at least $z$. For $p \ge y$ the quantity $1 - \log p/\log y$ is at most $0$, hence
\[
    \begin{split}
        \sum_{\substack{p\mid n\\ p < y}}\biggl(1-\frac{\log p}{\log y}\biggr) &\ge \sum_{p \mid n}\biggl(1-\frac{\log p}{\log y}\biggr) = \omega(n)-\frac{\log n}{\log y} \\
        &\ge \omega(n) - \frac{\log x}{\log y} = \omega(n)-\frac{20}{3}.
    \end{split}
\]
Therefore $b_n > 0$ forces $\omega(n)-20/3<\lambda = 4/3$, that is $\omega(n) < 8$. Since $b_n\le1$ for every $n$, it follows that
\[
    \sum_{\substack{n\le x \\ n\equiv3 \pmod 4\\ \mu^2(n) = 1 \\ \omega(n) \le 7}} r(n-2)r(n+2) \ge \sum_{\substack{n\le x\\ (n,P(z)) = 1\\ \mu^2(n) = 1}}a_n b_n \ge W(x) -\sum_{\substack{n\le x\\ (n,P(z))=1\\ \mu^2(n) = 0}} a_n|b_n| .
\]
For $n\le x$ with all prime factors at least $z$ we have
\[
    |b_n| \le 1 + \frac{\omega(n)}{\lambda}\le 1 + \frac{\log x}{\lambda\log z}\ll 1,
\]
and such an $n$ that is not square-free is divisible by $p^2$ for some prime $p$ with $z\le p\le\sqrt x$. Hence, by the trivial bound for $r$, the last sum above is $\ll x^{\varepsilon}(x/z+\sqrt x)=o(x/\log x)$. 
It therefore remains to prove that $W(x) \gg x/\log x$.

We are now ready to apply linear sieve, i.e. Lemma \ref{linear-sieve-lem} and bound $W(x)$.
Interchanging the order of summation we get
\[
    W(x) = S(\mathcal A,\mathcal P,z)-\frac{1}{\lambda} \sum_{z \le p< y}\biggl(1-\frac{\log p}{\log y}\biggr)S(\mathcal A_p,\mathcal P,z),
\]
where $\mathcal A_p=(a_n)_{p\mid n}$
and $S(\mathcal A_p,\mathcal P,z) = \sum_{p\mid n,\,(n,P(z))=1}a_n$.
Note that $p \ge z$, so the condition $(n,P(z)) = 1$ does not exclude $p \mid n$.

For the first term we apply Lemma \ref{linear-sieve-lem} with level $D = z^4$, so that $s=4$ and $f(4) = \tfrac{e^\gamma}{2}\log 3$, which gives
\[
    S(\mathcal A,\mathcal P,z)\ge XV(z)\biggl(\frac{e^\gamma}{2}\log 3 - B(\log D)^{-1/14} \biggr)-\sum_{\substack{d<D \\ d\mid P(z)}}|r_d| .
\]

For $z\le p<y$ and odd square-free $d$ coprime to $p$ we have
\[
    \sum_{\substack{n\le x\\ n\equiv 0 \pmod {pd}}} a_n = 8g(pd)x + r_{pd} = 8g(p)x g(d)+r_{pd},
\]
so the sequence $\mathcal A_p$ satisfies the hypotheses of Section \ref{linear-sieve-section} with $X_p = 8g(p)x$, the same density $g$, and remainders $r_{pd}$. Its level of distribution is $D_p = D/p$, and
\[
    s_p = \frac{\log D_p}{\log z} = 4\biggl(1-\frac{\log p}{\log D}\biggr)
\]
ranges from $4(1-3/(20\alpha)) < 1$ to $3$. 
Further, Lemma \ref{linear-sieve-lem} applies when $s_p \ge 1$, so we first work out the case $s_p < 1$. We have
\[
    S(\mathcal{A}_p, \mathcal{P},z) \le S(\mathcal{A}_p, \mathcal{P},D_p).
\]
By the sieve condition we have
\[
    \frac{V(D_p)}{V(z)} = \prod_{D_p \le p < z} \bigl(1-g(p) \bigr)^{-1} \le \frac{\log z}{\log D_p} \biggl( 1 + \frac{L}{\log D_p} \biggr) = s_p^{-1} (1+o(1)),
\]
where the error is uniform as $D_p \ge D/y \ge x^{1/100}$, therefore Lemma \ref{linear-sieve-lem} applies.

Now we have $s_p \ge 1$ and therefore can apply Lemma \ref{linear-sieve-lem}, the latter applies because $D_p \ge D/y\ge x^{1/100}$. In both cases
\[
    S(\mathcal A_p,\mathcal P,z) \le 8g(p)x V(z) \frac{2e^\gamma}{s_p}\,(1+o(1))+\sum_{\substack{d < D/p\\ d\ \mathrm{odd\ square\text{-}free}}}|r_{pd}| ,
\]
where the term $o(1)$, which comes from $B(\log(D/p))^{-1/14}$ and from $L/\log D_p$, is uniform in $p$ because $D_p\ge x^{1/100}$. Every odd square-free $m<D$ occurs as $m=pd$ for at most $\omega(m)\ll \log D$ primes $p$, so the total contribution of all the remainder terms above to $W(x)$ is $\ll x^{1-\varepsilon}\log x=o(x/\log x)$.

We proceed with summing over $p$.
Write $t = \log p/\log x$, so that the condition $z \le p< y$ becomes $\alpha/4 \le t < 3/20$, while $1/s_p = \alpha/(4(\alpha-t))$ and $1-\log p/\log y=1-20t/3$. 
Let $\phi$ be any function that is continuously differentiable on $[\alpha/4,3/20]$. 
Since $g(p)=1/p+O(p^{-2})$ and $\sum_{p\le v}1/p=\log\log v+c_0+O(1/\log v)$, partial summation gives
\[
    \sum_{z \le p < y}g(p) \phi \biggl(\frac{\log p}{\log x}\biggr) = \int_{\alpha/4}^{3/20} \phi(t)\frac{dt}{t} + O\biggl(\frac{1}{\log x}\biggr),
\]
the implied constant depending on $\phi$.
Take 
\[
    \phi(t) = \biggl(1-\frac{20}{3}t\biggr) \cdot 2e^\gamma\frac{\alpha}{4(\alpha-t)},
\]
which is smooth on the interval $[\alpha/4,3/20]$ since $\alpha-t \ge \alpha - 3/20 > 0$.
Then we have
\[
    \frac{1}{\lambda} \sum_{z \le p < y}\biggl(1-\frac{\log p}{\log y}\biggr) S(\mathcal A_p,\mathcal P,z)\le XV(z)\biggl(\frac{1}{\lambda}J(\alpha)+o(1)\biggr)+o\biggl(\frac{x}{\log x}\biggr),
\]
where
\[
    J(\alpha) = \frac{e^\gamma\alpha}{2}\int_{\alpha/4}^{3/20}\frac{1-20t/3}{t(\alpha-t)}dt.
\]
Setting $\eta = \log y / \log x = 3/20$ inside the integral we have 
\[
    \frac{\alpha(1-t/\eta)}{t(\alpha-t)} = \frac{1}{t} - \frac{\alpha-\eta}{\eta}\cdot\frac{1}{\alpha-t},
\]
hence
\[
    J(\alpha) = \frac{e^\gamma}{2}\biggl(\log\frac{4\eta}{\alpha}-\frac{\alpha-\eta}{\eta}\log\frac{3\alpha}{4(\alpha-\eta)}\biggr).
\]
Both terms depend continuously on $\alpha$ near $1/6$, and for $\alpha=1/6$, that is for $\varepsilon=0$, we have $4\eta/\alpha=18/5$, $(\alpha-\eta)/\eta=1/9$ and $3\alpha/(4(\alpha-\eta))=15/2$, whence
\[
    J(1/6) = \frac{e^\gamma}{2}\biggl(\log\frac{18}{5}-\frac19\log\frac{15}{2}\biggr).
\]

Combining the splitting of $W(x)$ with the two bounds just obtained, we get
\[
    W(x) \ge XV(z)\biggl(\frac{e^\gamma}{2}\log3-\frac34J(\alpha)-o(1)\biggr)-o\biggl(\frac{x}{\log x}\biggr).
\]
At $\alpha=1/6$ the constant becomes
\[
    \begin{split}
        \frac{e^\gamma}{2}\biggl(\log 3-\frac{3}{4} \log\frac{18}{5}+\frac{1}{12}\log\frac{15}{2}\biggr)
        &=\frac{e^\gamma}{24}\log\frac{3^{12}\cdot(15/2)}{(18/5)^9} > 0.
    \end{split}
\]
By continuity in $\alpha$ we may fix $\varepsilon > 0$ so small that
\[
    \frac{e^\gamma}{2}\log3-\frac34J(\alpha)>\frac{5e^\gamma}{48}\log\frac{25}{12},
\]
and then, for large $x$,
\[
    W(x) \ge\frac{5e^\gamma}{48}\log\frac{25}{12}\cdot XV(z) \gg \frac{x}{\log x},
\]
since $V(z)\asymp 1/\log z$.
This gives
\[
    \sum_{\substack{n\le x \\ n\equiv 3\pmod 4\\ \mu^2(n) = 1 \\ \omega(n) \le 7}} r(n-2)r(n+2) \gg \frac{x}{\log x},
\]
and by \eqref{sqfree-decomp} the left-hand side is at most $2s_{\Gamma}^{(7)}(x)$, which proves Theorem \ref{almostprimes-theorem} for large $x$. 
For bounded $x$ the desired conclusion follows trivially from $s^{(N)}_{\Gamma}(x) \ge s^{(N)}_{\Gamma}(3) > 0$.

\subsection*{Acknowledgements}

The author would like to thank Andrew Granville, Dimitris Koukoulopoulos and Igor Shparlinski for insightful discussions at various stages of this work.

\bibliographystyle{abbrv}

\bibliography{bibliography}{}

\end{document}